\documentclass[letterpaper, 10 pt, conference]{ieeeconf}  

\IEEEoverridecommandlockouts                              
\usepackage{amsmath}		
\usepackage{amscd}
\usepackage{amsfonts}
\usepackage{amssymb}		
\usepackage{mathtools}		

\usepackage[english]{babel}		
\usepackage[utf8]{inputenc}	
\usepackage{calc}
\usepackage{thumbpdf}			
\usepackage{latexsym}
\usepackage{xcolor}
\usepackage{enumerate}		
\usepackage{url}				
\usepackage{cite}	
\usepackage{graphics} 
\usepackage{epsfig} 
\usepackage{color}

\usepackage[bookmarks,bookmarksnumbered,colorlinks]{hyperref}
\hypersetup{linkcolor = black,anchorcolor = black,citecolor =
	black,filecolor = black,urlcolor = black}

\definecolor{darkgreen}{rgb}{0,0.6,0}

\newtheorem{defn}{Definition}

\newtheorem{thm}[defn]{Theorem}

\newtheorem{example}[defn]{Example}
\newtheorem{ass}[defn]{Assumption}
\newtheorem{rem}[defn]{Remark}

\usepackage{algorithm}
\usepackage[noend]{algpseudocode}

\DeclareMathOperator*{\argmin}{arg\,min}
\DeclareMathOperator*{\minimize}{minimize}
\def\eps{\varepsilon}

\newcommand{\Xstraj}[1]{X^s(#1)}

\newcommand{\R}{\mathbb{R}}
\newcommand{\N}{\mathbb{N}}

\newcommand{\K}{\mathcal{K}}
\newcommand{\U}{\mathcal{U}}
\newcommand{\X}{\mathcal{X}}
\newcommand{\XX}{\mathbb{X}}

\newcommand{\W}{\mathcal{W}}

\newcommand{\RR}[1]{\mathcal{R}(#1)}

\newcommand{\Llp}[2][p]{L^{#1}(\Omega,\F,\mathbb{P};#2)}

\newcommand{\F}{\mathcal{F}}

\newcommand{\Prob}{\mathbb{P}}
\newcommand{\E}{\mathbb{E}}

\newcommand{\T}{\mathbb{T}}
\newcommand{\Exp}[1]{\mathbb{E}\left[#1\right]}

\title{\LARGE \bf
Towards turnpike-based performance analysis of risk-averse stochastic predictive control
}

\author{Jonas Schießl, Ruchuan Ou, Michael H. Baumann, Timm Faulwasser, and Lars Grüne
\thanks{
The authors gratefully acknowledge that this work was funded by the Deutsche Forschungsgemeinschaft (DFG, German Research Foundation) – project number 499435839.}
\thanks{Jonas Schießl, Michael H. Baumann and Lars Grüne are with Mathematical Institute, University of Bayreuth, Germany,
        \tt\small \{jonas.schiessl,michael.baumann,lars.gruene\} @uni-bayreuth.de}%
\thanks{Ruchuan Ou and Timm Faulwasser are with the Institute of Control Systems, Hamburg University of Technology, Hamburg, Germany,
        \tt\small ruchuan.ou@tuhh.de, timm.faulwasser@ieee.org}%
}

\begin{document}

\maketitle
\pagestyle{empty}

\begin{abstract}
    In this paper, we present performance estimates for stochastic economic MPC schemes with risk-averse cost formulations. 
    For MPC algorithms with costs given by expectations, it was recently shown that the guaranteed near-optimal performance of abstract MPC in random variables coincides with its implementable variant using pathwise feedback.
    In general, this property does not extend to costs formulated in terms of risk measures.
    However, through a turnpike-based analysis, this paper demonstrates that for a particular class of risk measures, this result can still be leveraged to formulate an implementable risk-averse MPC scheme, resulting in near-optimal averaged performance.
\end{abstract}

\section{Introduction}\label{sec:1}
In optimal control and model predictive control (MPC), risk can be considered in two fundamentally different ways: by relaxing inequality constraints to hold with a given probability or by ensuring that unlikely outcomes do not lead to arbitrary bad performance. While both risk concepts appears in stochastic MPC, this paper is concerned with the latter, i.e., we focus on risk-averse objective transformations in stochastic MPC. 

When formulating stochastic optimal control problems, one cannot simply evaluate the deterministic objective function with random variable arguments as this will give a random variable. Instead, one has to map random variables to real numbers. The most frequently, used approach is optimization in expectation.  However, it is known that risk measures such as, e.g., the averaged value-at-risk, also known as conditional value-at-risk or expected shortfall, are better suited to avoiding rare outcomes with bad performance~\cite{rockafellar2007coherent,shapiro2021lectures}.
Hence, the consideration of risk-aware objective transformations has been addressed in stochastic optimal control~\cite{Guigues2023, Isohatala2021, pichler2023risk, ahmadi2023risk} and has also gained attention in stochastic predictive control.
In particular, \cite{tooranjipour2024risk,yin2023risk,venkatasubramanian2020stochastic,bemporad2011stochastic} employ conditional value-at-risk to take risk into account, \cite{singh2018framework} uses polytopic risk measures, and \cite{sopasakis2019risk} considers coherent measures of risk on a discrete sample space. 

Considering expected value objectives, we have recently extended dissipativity notions for optimal control problems to the stochastic setting and we analyzed different kinds of stochastic turnpike notions and their relation to each other~\cite{Schiessl2023a,Schiessl2024a,Schiessl2025}. The underlying motivation is that turnpike and dissipativity properties play a crucial role in the analysis of deterministic MPC schemes~\cite{grune2022dissipativity}. 
Moreover, in \cite{Schiessl2024c} we have given novel performance estimates for stochastic MPC. Yet, all of our above works do not take into account risk aware objective formulations. 

Hence, in this work, we analyze the performance of risk-aware stochastic MPC. In particular, we derive conditions under which a risk-aware abstract MPC in random variables delivers performance close to the optimal performance of a stationary process. To this end and similar to \cite{BenTal1987,Shapiro2017,Guigues2023} we rely on parametrized risk measures.

The remainder of this paper is structured as follows:
In Section \ref{sec:2} we introduce the considered stochastic problem setting and an illustrative example, while in Section \ref{sec:3} we present our main results on risk aware stochastic MPC. The paper ends with conclusions in Section \ref{sec:conclusions}.

\section{Setting and preliminaries}\label{sec:2}
\subsection{Problem formulation}
    We consider discrete-time stochastic system
    \begin{equation} 
    \label{eq:stochSys}
        X(k+1) = f(X(k),U(k),W(k)), \quad X(0) = X_0.
    \end{equation}
    defined by a continuous function
    \begin{equation*}
        f : \X\times \U \times \W \rightarrow \X, \quad (x,u,w) \mapsto f(x,u,w).
    \end{equation*}
    Here $\X$, $\U$, and $\W$ are Borel spaces and the initial condition $X_0 \in \RR{\Omega,\X}$, the states $X(k) \in \RR{\Omega,\X}$, the controls $U(k) \in \RR{\Omega,\U}$ as well as the noise $W(k) \in \RR{\Omega,\W}$ are random variables on the probability space $(\Omega, \mathcal{F}, \Prob)$ for all $k \in \N_0$, where 
    $\RR{\Omega,\X} := \{X: (\Omega, \mathcal{F},\Prob) \rightarrow \X \mbox{ measurable} \}$.
    Furthermore, $W(k)$ is independent of $X(k)$ and $U(k)$ for all $k \in \N_0$ and the sequence $\{W(k)\}_{k \in \N_0}$ is \emph{i.i.d.} with known distribution $P_W$.
    
    Additionally, we assume that the control process $\mathbf{U} := (U(0),U(1),\ldots)$ is measurable with respect to the natural filtration $(\F_k)_{k \in \N_0}$, i.e.
    \begin{equation} \label{eq:Filtration}
        \sigma(U(k)) \subseteq \F_k := \sigma((X(0),\ldots,X(k)) \subseteq \F
    \end{equation}
    for all $k \in \N_0$. This condition can be seen as a causality requirement, formalizing that we do not use information about the future noise and only the information contained in $X(0),\ldots,X(k)$ about the past noise when deciding about our control values. For more details on stochastic filtrations we refer to \cite{Fristedt1997, Protter2005}. 
    
    We call a control sequence that satisfies \eqref{eq:Filtration} admissible and denote the set of all admissible control sequences for the initial value $X_0$ on horizon $N \in \N \cup \{\infty\}$ by $\U^N(X_0)$.
    For a given initial value $X_0$ and control sequence $\mathbf{U}$, we denote the solution of system \eqref{eq:stochSys} by $X_{\mathbf{U}}(\cdot,X_0)$, or short by $X(\cdot)$ if the initial value and the control are unambiguous. 
    Note, that the solution $X_{\mathbf{U}}(\cdot,X_0)$ also depends on the disturbance $\mathbf{W}$. 
    However, for the sake of readability, we do not highlight this in our notation and assume in the following that $\mathbf{W} := (W(0),W(1),\ldots)$ is an arbitrary but fixed stochastic process.

    Moving from the stochastic dynamics \eqref{eq:stochSys} to  optimal control problem, we consider the stage cost 
    \[\ell(X,U) := \T[g(X,U)]\] 
    where $g: \X \times \U \rightarrow \R$ is a continuous function bounded from below and $\T : \mathcal{Z} \rightarrow \R$ is a mapping of random variables from a linear space $\mathcal{Z}$ to the real values.
    
    Then, the stochastic optimal control problem under consideration reads
    \begin{equation} 
    \label{eq:stochOCP}
        \begin{split}
            \minimize_{\mathbf{U} \in \U^{N}(X_0)} &J_N(X_0,\mathbf{U}) := \sum_{k=0}^{N-1} \ell(X(k),U(k)) \\
            s.t. ~ X(k+1) &= f(X(k),U(k),W(k)), ~ X(0) = X_0.
        \end{split}
    \end{equation}
    By $V_N(X_0) := \inf_{\mathbf{U} \in \U^{N}(X_0)} J_N(X_0,\mathbf{U})$ we denote the optimal value function of the optimal control problem \eqref{eq:stochOCP} and if a minimizer of this problem exists we will denote it by $\mathbf{U}^*_N$ or $\mathbf{U}^*_{N,X_0}$ if we want to emphasize the dependence on the initial condition. 

    Moreover, in order to guarantee well-posedness of problem \eqref{eq:stochOCP} and finiteness of the optimal value function $V_N(X_0)$ for finite $N \in \N$, we assume the initial values $X_0$ --- and thus the whole optimal trajectory $X_{\mathbf{U}^*_{N}}(\cdot,X_0)$ --- to lie in the constraint set
    \begin{equation*}
    \begin{split}
        \mathbb{X} := \{ X &\in \RR{\Omega, \X} \mid ~\exists~ U \in \RR{\Omega, \X} : \\
        &\vert \ell(X,U) \vert < \infty, ~f(X,U,W) \in \mathbb{X}\} \subseteq \RR{\Omega, \X}.
    \end{split}
    \end{equation*}
    For instance, if we consider the generalized linear-quadratic problem from \cite{Schiessl2025} with $W(k) \in \Llp[2]{\W}$ we get $\mathbb{X} = \Llp[2]{\X}$.

    \begin{rem}
        Problem~\eqref{eq:stochOCP} can equivalently be formulated as a Markov decision problem (MDP) in terms of feedback policies and distributions rather than random variables.
        While optimizing over policies instead of control sequences yields no essential difference, a distribution-based formulation will prevent us from making statements about the realization paths of the solutions, cf.\ \cite{Schiessl2024a}.
        To clearly distinguish our approach from the MDP perspective, we therefore conduct the analysis in terms of control sequences.
    \end{rem}

\subsection{Fundamentals of risk-measures}
    A common choice for the mapping $\mathbb{T}[Z]$ in \eqref{eq:stochOCP} is the expected value $\E[Z]$. 
    However, as risk measures are widely used as optimization criteria in various risk-averse applications (see the references in the introduction), in this paper we will consider a class of risk measures as the mapping $\mathbb{T}[Z]$ in \eqref{eq:stochOCP}.

    \begin{defn}[Risk measures] \label{defn:riskMeasure}
        Let $\mathcal{Z}$ be a linear space of measurable functions $Z: (\Omega,\mathcal{F},\mathbb{P}) \rightarrow \R$, where $(\Omega,\mathcal{F},\mathbb{P})$ is a probability space.
        A function $\rho: \mathcal{Z} \rightarrow \R \cup \{\infty\}$ is called risk measure if it is
        \begin{enumerate}[(i)]
            \item \emph{monotone}, i.e.\ for all $Z_1, Z_2 \in \mathcal{Z}$ it holds that $Z_1 \geq Z_2 \Rightarrow \rho(Z_1) \geq \rho(Z_2)$ 
            \item \emph{translative}, i.e.\ for all $Z \in \mathcal{Z}$ and $m \in \R$ it holds that $\rho(Z + m) = \rho(Z) + m$ 
        \end{enumerate}
    \end{defn}

    \begin{rem}
        In our definition of a risk measure $\rho$ we associate high values of the random variable $Z$ with a high risk $\rho(Z)$. While  in insurance mathematics the same concept is used, in financial mathematics the common choice would be to associate low values of $Z$ with a high risk $\tilde{\rho}(Z)$. However, the two definitions are easily transformed into each other as $\rho(Z) = \tilde{\rho}(-Z)$.
    \end{rem}

    While an arbitrary risk measure only has to satisfy the conditions of Definition~\ref{defn:coherentRiskMeasure}, there are several additional properties for risk measures, which can be useful under certain circumstances.
    Especially, law-invariance will be necessary for our later investigations.
    
    \begin{defn}\label{defn:coherentRiskMeasure}
        A risk measure $\rho: \mathcal{Z} \rightarrow \R \cup \{\infty\}$ satisfying  Definition~\ref{defn:riskMeasure} is called
        \begin{enumerate}[(i)]
            \item \emph{convex} if $\rho(\beta Z_1 + (1-\beta) Z_2) \leq \beta \rho(Z) + (1-\beta) \rho(Z)$ for all $Z_1,Z_2 \in \mathcal{Z}$, $\beta \in (0,1)$;
            \item \emph{positive homogeneous} if $\rho(\beta Z) = \beta \rho(Z)$ for all $Z \in \mathcal{Z}$ and $\beta \geq 0$;
            \item \emph{coherent} if it is convex and positive homogeneous;
            \item \emph{law-invariant} if $\rho(Z_1) = \rho(Z_2)$ holds for all $Z_1,Z_2 \in \mathcal{Z}$ with $Z_1 \sim Z_2$.
        \end{enumerate}
    \end{defn}

    For coherent risk measures, the robust representation theorem, cf.\ \cite{Follmer2011}, can be used to represent them in terms of expectations as
    \begin{equation} \label{eq:robustRep}
        \rho(Z) = \sup_{\mathbb{Q} \in \mathcal{Q}} \E^{\mathbb{Q}}[Z]
    \end{equation}
    for some set $\mathcal{Q} \subseteq \mathcal{M}(\Prob) := \{ \mathbb{Q} \ll \Prob \}$ of probability measures which are absolutely continuous to $\mathbb{P}$.
    However, using this representation can be challenging in applications since one has to take the supremum over probability measures. 
    Thus, for our approach, we consider another special class of risk measures which have been considered in \cite{BenTal1987, Shapiro2017} and \cite{Guigues2023} in an optimal control contexts and which can be written in  parametric form.

    \begin{defn}[parameterized risk measures] \label{defn:riskMeasureParametric}
        A risk measure $\rho: \mathcal{Z} \rightarrow \R \cup \{\infty\}$ according to Definition~\ref{defn:riskMeasure} is called \emph{parameterized} if there exists a set $\Theta \subset \R^q$, $q \in \N$ and a function $\Psi: \R \times \Theta \rightarrow \R$ such that 
        \begin{equation} \label{eq:parametricForm}
            \rho(Z) = \inf_{\theta \in \Theta} \E[\Psi(Z,\theta)].
        \end{equation}
    \end{defn}

    Note that for every proper, closed and monotone non-decreasing function $\psi: \R \rightarrow \R \cup \{\infty\}$, the function
    \begin{equation} \label{eq:simpleParametricRiskMeasure}
        \rho(Z) := \inf_{\theta \in \R} \E[\theta + \psi(Z - \theta)]
    \end{equation}
    defines a risk measure according to Definition~\ref{defn:riskMeasure} given in parametric form \eqref{eq:parametricForm}. 
    Moreover, if $\psi$ is additionally convex the risk measure \eqref{eq:simpleParametricRiskMeasure} is also convex. 
    However, it is in general not positive homogeneous, and thus, not coherent.
    Several risk measures that are used in practice can be represented in the parametric form \eqref{eq:parametricForm}, as shown in the following examples.

    \begin{example}[Averaged value-at-risk] \label{example:AV@R}
        The averaged value-at-risk (also called expected shortfall or conditional value-at-risk, cf.\ Remark~\ref{rem:riskMeasures}~(i)) with confidence-level $\alpha \in (0,1]$ is a coherent risk measure and can be parametrized as
        \begin{equation} \label{eq:AV@R}
            \rho(Z) = \inf_{\theta \in \R}\E \left[ \theta + \alpha^{-1} \max\{ 0,Z-\theta\} \right].
        \end{equation}
    \end{example}

    \begin{example}[$\phi$-divergence risk measures] \label{example:phiDivergence}
         Let $\mathcal{D}(\Omega) := \{ \xi: \Omega \rightarrow \R^+_0 \mid \int_{\Omega} \xi d\mathbb{P} = 1 \} \subseteq \Llp[1]{\R}$ be the space of probability density functions on $\Omega$ and define the $\phi$-divergence ambiguity set
        \begin{equation} 
            \mathcal{A} := \left\{ \xi \in \mathcal{D}(\Omega): \int_{\Omega} \phi(\xi(\omega)) d\mathbb{P}(\omega) \leq c \right\}, \quad c \geq 0,
        \end{equation}
        where $\phi: \R \rightarrow \R^+_0 \cup \{+\infty\}$ is a convex lower semicontinuous function with $\phi(1)=0$ and $\phi(z) = +\infty$ for $z < 0$. Then, the coherent risk measure corresponding to $\mathcal{A}$ with constraint-level $c \geq 0$ is given by $\rho(Z) = \sup_{\xi \in \mathcal{A}} \int_{\Omega} \xi Z d\mathbb{P}$ and by duality arguments this risk measure can be written in parametric form \eqref{eq:parametricForm} as
        \begin{equation}
            \rho(Z) = \inf_{\theta \in \R^+ \times \R}\E \left[ \theta_1 c + \theta_2 + \theta_1 \phi^*\left(\frac{Z-\theta_2}{\theta_1}\right) \right],
        \end{equation}
        where $\phi^*$ is the Legendre-Fenchel conjugate of $\phi$. 
        In particular, by choosing $\phi$ as the Kullback-Leibler-divergence $\phi(z) := z \ln(z) - z +1$ for $z \geq 0$ and $\phi(x) := +\infty$ for $z < 0$ we obtain the risk measure
        \begin{equation} \label{eq:risk_KL}
            \rho(Z) = \inf_{\theta \in \R^+ \times \R}\E \left[ \theta_1 c + \theta_2 + \theta_1 e^{\left(\frac{Z-\theta_2}{\theta_1}\right)} - \theta_1 \right].
        \end{equation}
        For more details on this we refer to \cite{Shapiro2017,BenTal1987}.
    \end{example}

    \begin{rem} \label{rem:riskMeasures}
        \begin{enumerate}[(i)]
            \item While expected shortfall is used as a synonym for averaged value-at-risk, the term conditional value-at-risk is used ambiguously in the literature. Depending on the reference, it is either used again as a synonym for averaged value-at-risk, cf.\ \cite{rockafellar2007coherent}, or as another name for tail conditional expectation. While these two concepts coincide for continuous random variables, this is in general not the case for discrete ones, cf.\ \cite{acerbi2002spectral}.
            \item By choosing $\psi(z) = z$ in \eqref{eq:simpleParametricRiskMeasure} it follows that $\rho(Z) = \Exp{Z}$ is a parametrizeable risk measure in the sense of Definition~\ref{defn:riskMeasureParametric}, and thus, is also included in our considerations.
        \end{enumerate}
    \end{rem}
    
\subsection{Turnpike properties in stochastic optimal control}
    Our analysis in Section~\ref{sec:3} will be based on recently developed turnpike properties for stochastic optimal control problems, see \cite{Schiessl2023a,Schiessl2024a,Schiessl2025}, which we recall in the following.
    While in deterministic settings the turnpike is usually an equilibrium of the system, in stochastic settings this is not possible due to the persistent excitation of the system~\eqref{eq:stochSys} by the noise. 
    Thus, we make the following definition of a stationary solution of system~\eqref{eq:stochOCP}.

    \begin{defn}[Stationary stochastic processes] \label{defn:stationaryProcess}
        A pair of the stochastic processes $(\mathbf{X}^s,\mathbf{U}^s)$ given by
        \begin{equation} \label{eq:sys_stat}
            \Xstraj{k+1} = f(\Xstraj{k}, U^s(k), W(k))
        \end{equation}
        with $\mathbf{U}^s \in \U^{\infty}(X^s(0))$ is called stationary for system \eqref{eq:stochSys} if there exist probability distributions $P^s_X$, $P^s_U$, and $P^s_{X,U}$ with
        \begin{equation*}
        \begin{split}
            \Xstraj{k} \sim P^s_X, \quad U(k) \sim P^s_U, \quad (\Xstraj{k},U^s(k)) \sim P^s_{X,U}
        \end{split}
        \end{equation*}
        for all $k \in \N_0$. 
    \end{defn}

    The next definition formalizes the turnpike property of stochastic optimal control problem~\eqref{eq:stochOCP} where the turnpike is defined as a stationary solution from Definition~\ref{defn:stationaryProcess}. Note that this definition is equivalent to the definition used in \cite{Schiessl2024a}, if we set $\XX_0 = \{X_0\}$ and considering only optimal trajectories instead of near stationary solutions.

    \begin{defn}[Random variable turnpike] \label{defn:stochTurnpike}
        Consider a stationary pair $(\mathbf{X}^s,\mathbf{U}^s)$, a pseudometric $d$ on the space $\RR{\Omega,\X}$ and a set $\tilde{\XX} \subseteq \XX$. 
        We say that the optimal control problem \eqref{eq:stochOCP} has
        \begin{enumerate}[(i)]
            \item a stochastic finite horizon turnpike property on $\tilde{\XX}$ if there exists $\vartheta \in \mathcal{L}$ such that for each optimal trajectory $X_{\mathbf{U}^*_{N}}(\cdot,X_0)$ with $X_0 \in \tilde{\XX}$ and all $N,L \in \N$ there is a set $\mathcal{Q}(X_0,L,N) \subseteq \{0,\ldots,N\}$ with $\# \mathcal{Q}(X_0,L,N) \leq L$ elements and 
            \begin{equation*}
                d \left( X_{\mathbf{U}^*_{N}}(k,X_0) , X^s(k) \right) \leq \vartheta(L)
            \end{equation*}
            for all $k \in \{0,\ldots,N\} \setminus \mathcal{Q}(X_0,L,N)$;
            \item a stochastic infinite horizon turnpike property on $\tilde{\XX}$ if there exists $\vartheta_{\infty} \in \mathcal{L}$ and $r \in \R^+$ such that for each optimal trajectory $X_{\mathbf{U}^*_{\infty}}(\cdot,X_0)$ with $X_0 \in \tilde{\XX}$ and all $L \in \N$ there is a set $\mathcal{Q}(X_0,L,\infty) \subseteq \N_0$ with $\# \mathcal{Q}(X_0,L,\infty) \leq L$ elements and 
            \begin{equation*}
                d \left( X_{\mathbf{U}^*_{\infty}}(k,X_0) , X^s(k) \right) \leq \vartheta_{\infty}(L)
            \end{equation*}
            for all $k \in \N_0 \setminus \mathcal{Q}(X_0,L,\infty)$.
        \end{enumerate}
    \end{defn}

    \begin{defn}[Types of stochastic turnpike properties] \label{defn:stochTurnpikeTypes}
        Assume that the optimal control problem has a stochastic turnpike property according to definition \ref{defn:stochTurnpike}. Then we say that it has 
        \begin{enumerate}[(i)]
            \item the \emph{$L^r$ turnpike property} if $d$ is the $L^r$-norm, i.e.\ $$d(X,Y) = \Vert X - Y \Vert_{L^r} := (\Exp{\Vert X - Y \Vert^r})^{1/r}.$$
            \item the \emph{pathwise-in-probability turnpike property} if $d$ is the Ky-Fan metric, i.e.\ $$d(X,Y) = d_{KF}(X,Y) := \inf_{\eps > 0} \{\Prob(\Vert X - Y \Vert > \eps) \leq \eps \}.$$
            \item the \emph{distributional turnpike property} if $d$ is a metric on the space of probability measures $\mathcal{P}(\X)$, or more precisely
            \begin{enumerate}[(a)]
                \item the \emph{Wasserstein turnpike property of order $r$} if $d$ is the Wasserstein distance of order $r$, i.e\ $$d(X,Y) = \inf \{ \Vert \bar{X} - \bar{Y} \Vert_{L^r} \mid \bar{X} \sim X, \bar{Y} \sim Y \}.$$
                \item the \emph{weak distributional turnpike property} if $d$ is the Lévy-Prokhorov metric, i.e.\ $$d(X,Y) = \inf \{ d_{KF}(\bar{X}, \bar{Y}) \mid \bar{X} \sim X, \bar{Y} \sim Y \}.$$
            \end{enumerate}
            \item the \emph{$r$-th moment turnpike property} if $$d(X,Y) = \left\vert \Exp{\Vert X \Vert^r}^{1/r} - \Exp{\Vert Y \Vert^r}^{1/r} \right\vert.$$
        \end{enumerate}
    \end{defn}

    \begin{example} \label{ex:example1}
        Consider the stochastic optimal control problem 
        \begin{equation} \label{eq:example1}
            \begin{split}
                \minimize_{\mathbf{U} \in \U^{N}(X_0)} J_N(X_0,\mathbf{U}) &:= \sum_{k=0}^{N-1} \mathbb{T}[X(k)^2 + 5 U(k)^2] \\
                s.t. ~ X(k+1) &= 1.5 X(k) + U(k) + W(k)
            \end{split}
        \end{equation}
        where $W(k)$ follows a two-point distribution such that $W(k) = a := 0.6$ with probability $p_a = 0.5$ and $W(k) = b := -0.6$ with probability $p_b = 0.5$.
        For $\mathbb{T}[Z] = \E[Z]$ this corresponds to a standard stochastic linear-quadratic problem and for this problem it was already shown analytically and numerically that the problem exhibits stochastic turnpike properties in various forms, cf.\ \cite{Schiessl2024c, Ou2021}. 
        The optimal trajectories of problem \eqref{eq:example1} with $\mathbb{T}[Z] = \E[Z]$ for horizons $N=3,5,7,9,13$ and $X_0=1.5$ are shown in the left column of Figure~\ref{fig:example1_turnpike}.
        If, instead, we consider a risk measure as the mapping $\mathbb{T}[Z]$, then we can still observe stochastic turnpike properties as shown in the right column of Figure~\ref{fig:example1_turnpike} for the averaged value-at-risk from equation \eqref{eq:AV@R} of Example~\ref{example:AV@R} with confidence-level $\alpha=0.05$.
        Here, for numerical purpose we used the softplus function $\log(1+e^z)$ in the simulations instead of the function $\max\{0,z\}$ in \eqref{eq:AV@R}. 
        Note that usage of the softplus leads to an approximation of the averaged value-at-risk but can also be interpreted as a risk measure on its own sake due to equation \eqref{eq:simpleParametricRiskMeasure} with $\psi(z) = \log(1+e^z)$.
        The results from Figure~\ref{fig:example1_turnpike} show that although we can observe the same qualitative behavior for the expectation and the risk measure used in the cost, the overall width of the possible paths for the state trajectories is smaller if we use the risk-averse formulation rather than the expectation. 
        \begin{figure}[t]
            \centering
            \includegraphics[width=0.33\textwidth]{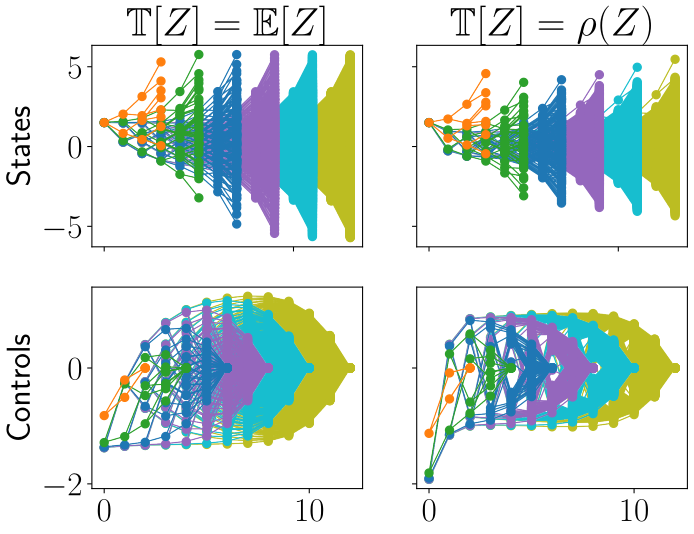}
            \caption{Optimal state and control trajectories on horizons $N=3,5,7,9,11,13$ for $X_0=1.5$ and $\mathbb{T}[Z] = \rho(Z)$ from equation \eqref{eq:AV@R} (right) as well as $\mathbb{T}[Z] = \E[Z]$ (left).}
            \label{fig:example1_turnpike}
        \end{figure}
    \end{example}
    \begin{example} \label{ex:example2}
        Consider the stochastic optimal control problem 
        \begin{equation} \label{eq:example2}
            \begin{split}
                \minimize_{\mathbf{U} \in \U^{N}(X_0)} J_N(X_0,\mathbf{U}) &:= \sum_{k=0}^{N-1} \mathbb{T}[X(k)^2 + \gamma U(k)^2] \\
                s.t. ~ X(k+1) &= (U(k)-X(k))^2 + W(k)
            \end{split}
        \end{equation}
        where $\gamma \geq 0$ is a regularization parameter and $W(k)$ follows a two-point distribution such that $W(k) = a := 1$ with probability $p_a = 0.7$ and $W(k) = b := 0.25$ with probability $p_b = 0.3$.
        For $\mathbb{T}[X] = \E[X]$ this is a slight modification of the example in Section V of \cite{Schiessl2024c}, where it was shown numerically that the problem exhibits stochastic turnpike phenomena for $\gamma > 0$.
        Moreover, for $\gamma = 0$ it was shown analytically in \cite{Schiessl2024a} that the problem has all the stochastic turnpike properties in Definition~\ref{defn:stochTurnpikeTypes}.
        For $\mathbb{T}[Z] = \E[Z]$ and $\mathbb{T}[Z] = \rho(Z)$, where $\rho(Z)$ is the Kullback-Leibler divergence from equation \eqref{eq:risk_KL} of Example~\ref{example:phiDivergence} with constraint level $c=0.5$, the state and control trajectories on horizons $N=3,5,7,9,11$ with $\gamma=15$ and $X_0=1.5$ are shown in Figure~\ref{fig:example2_turnpike}.
        We can again observe stochastic turnpike properties if we chose a risk measure as the mapping $\mathbb{T}[Z]$ in \eqref{eq:example2}.
        Moreover, we can observe that the qualitative behavior of the solutions remains the same for $\mathbb{T}[Z]=\E[X]$ or $\mathbb{T}[Z] = \rho(Z)$.
        However, the usage of the risk measure in the costs seems to lead to a consolidation of the paths such that high values for the states are less likely.
        \begin{figure}[t]
            \centering
            \includegraphics[width=0.33\textwidth]{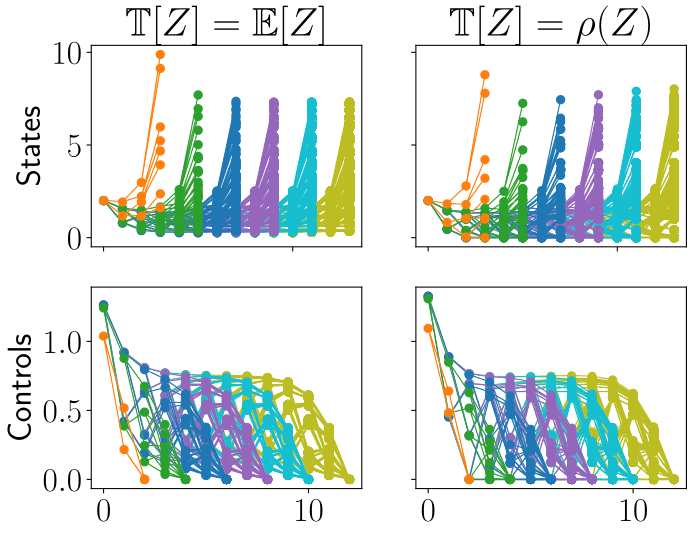}
            \caption{Optimal state and control trajectories on horizons $N=3,5,7,9,11,13$ for $X_0=1.5$ and $\mathbb{T}[Z] = \rho(Z)$ from equation \eqref{eq:risk_KL} (right) as well as $\mathbb{T}[Z] = \E[Z]$ (left).}
            \label{fig:example2_turnpike}
        \end{figure}
    \end{example}
    
\section{Risk-averse stochastic MPC}\label{sec:3}
    In this section, we aim to derive a risk-averse stochastic MPC algorithm that can also be implemented using measurements from a real plant. To this end, we start with Algorithm~\ref{alg:stochMPCabstract}, which implements a deterministic MPC algorithm for the stochastic problem formulated as a deterministic problem on the space of random variables $\RR{\Omega,\X}$.

    \begin{algorithm}
        \caption{Abstract stochastic MPC algorithm}\label{alg:stochMPCabstract}
        \begin{algorithmic}
            \For{$j=0,\ldots,K$}
                \State 1.) Set $X_0 = X(j)$.
                \State 2.) Solve the stochastic optimal control problem \eqref{eq:stochOCP}. 
                \State 3.) Apply the MPC feedback $\mu_N(X(j)) := U^*_{N,X_0}(0)$
                \State ~~~ to system \eqref{eq:stochSys} and get the next state $X(j+1)$.
            \EndFor
        \end{algorithmic}
    \end{algorithm}
 Algorithm~\ref{alg:stochMPCabstract} is useful for theoretical analysis, since we can transfer the proof ideas from the deterministic setting to the stochastic one. Yet, it requires knowledge of the complete random variables, or at least their distributions.
    However, in an implementation with a real plant, one cannot measure the random variable $X(j)$, but only a realization $x_j = X(j, \omega)$ thereof, i.e.\ a step of a single path drawn randomly from all possible paths. 
    Thus, an implementable version of Algorithm~\ref{alg:stochMPCabstract}, which is given by Algorithm~\ref{alg:stochMPCimpl}, would use the measured realization as initialization in each iteration instead of the random variable.

    \begin{algorithm}
        \caption{Implementable stochastic MPC algorithm}\label{alg:stochMPCimpl}
        \begin{algorithmic}
            \For{$j=0,\ldots,K$}
                \State 1.) Measure the state $x_j := X(j,\omega)$ and set $X_0\equiv x_j$.
                \State 2.) Solve the stochastic optimal control problem \eqref{eq:stochOCP}. 
                \State 3.) Apply the MPC feedback $\mu_N(x_j) := U^*_{N,x_j}(0)$
                \State ~~~ to system \eqref{eq:stochSys}.
            \EndFor
        \end{algorithmic}
    \end{algorithm}

    The following Theorem, which is \cite[Corollary~7]{Schiessl2024c}, shows that the closed-loop performance
    \[ J^{cl}_K(X_0, \mu) := \sum_{k=0}^{K-1} \ell(X_{\mu}(k,X_0),\mu(X_{\mu}(k,X_0))), \quad K \in \N \]
    coincides for the two algorithms if one uses the expected value as the mapping $\mathbb{T}$ in \eqref{eq:stochOCP}.

    \begin{thm} \label{thm:MPCalg}
        Consider $N,j,K \in \N$ and assume that the stage cost is defined as $\ell(X,U) = \E[h(X,U)]$ for some continuous function ${h: \X \times \U \rightarrow \R}$, which is bounded from below. 
        Let $\mu_N^2$ be a measurable feedback law from Algorithm~\ref{alg:stochMPCimpl} for $x_j=X(j,\omega)$. 
        Then $\mu_N^2$ coincides $P_{X(j)}$-almost surely with a feedback law $\mu_N^1$ from Algorithm~\ref{alg:stochMPCabstract} and the identity
        $ J^{cl}_K(X_0, \mu_N^1) =  J^{cl}_K(X_0, \mu_N^2)$
        holds for all $X_0 \in \XX$.  
    \end{thm} 

    Unfortunately, if we choose a mapping $\mathbb{T}$ in \eqref{eq:stochOCP} that is different from the expected value, the result of Theorem~\ref{thm:MPCalg} is, in general, no longer valid. 
    The reason for this is that a general map $\mathbb{T}$ will in general not fulfill the tower property of conditional expectations, which is crucial for proving Theorem~\ref{thm:MPCalg}.
    Therefore, although for general $\mathbb{T}$ we can still analyze the performance of the abstract Algorithm~\ref{alg:stochMPCabstract}, we cannot transfer these performance bounds to the implementable Algorithm~\ref{alg:stochMPCimpl}.
    However, if we consider a parameterized risk measure, cf.\ Definition~\ref{defn:riskMeasureParametric}, we can fix the parameter $\theta \in \Theta$ and consider the stage costs  $\ell_{\theta}(X,U) := \Exp{\Psi(g(X,U),\theta)}$, for which by optimality we get 
    \begin{equation*}
        \ell_{\theta}(X,U) \geq \inf_{\theta \in \Theta} \Exp{\Psi(g(X,U),\theta)} = \ell(X,U)
    \end{equation*}
    for all $\theta \in \Theta$ and to which we can apply Theorem~\ref{thm:MPCalg}. To formalize this, we make the following assumption.
    
    \begin{ass} \label{ass:mappingT}
    \begin{enumerate}[(i)]
        \item The mapping $\mathbb{T}$ from \eqref{eq:stochOCP} is a law-invariant parameterized risk measure with $\mathbb{T}[Z] = \inf_{\theta \in \Theta}\E[\Psi(Z,\theta)]$ and $\Psi(\cdot,\theta)$ bounded from below for all $\theta \in \Theta$.
        \item The function $\Psi$ satisfies the following uniform continuity assumption: for all $\theta_1 \in \Theta$ there exists $\alpha \in \K_{\infty}$ such that
        $\vert \Psi(z,\theta_1) - \Psi(z,\theta_2) \vert \leq \alpha( \Vert \theta_1 - \theta_2 \Vert)$
        holds for all $z \in \R$ and $\theta_2 \in \Theta$.
    \end{enumerate}
    \end{ass}
    
    Under Assumption~\ref{ass:mappingT} we can replace the original optimal control problem \eqref{eq:stochOCP} by the optimal control problem 
    \begin{equation} 
    \label{eq:stochOCPrisksensitive}
        \begin{split}
            \minimize_{\mathbf{U} \in \U^{N}(X_0)} &J_N(X_0,\mathbf{U}) := \sum_{k=0}^{N-1} \ell_{\theta}(X(k),U(k)) \\
            s.t. ~ X(k+1) &= f(X(k),U(k),W(k)), ~ X(0) = X_0
        \end{split}
    \end{equation}
    in each iteration of Algorithm~\ref{alg:stochMPCimpl}. This way we obtain Algorithm~\ref{alg:stochMPCrisk}, where in Step~2.) an upper bound of the original problem is minimized and the performance again coincides with the abstract version of this algorithm.

    \begin{algorithm}
        \caption{risk-averse stochastic MPC algorithm}\label{alg:stochMPCrisk}
        \begin{algorithmic}
            \For{$j=0,\ldots,K$}
                \State 1.) Measure the state $x_j := X(j,\omega)$ and set $X_0\equiv x_j$.
                \State 2.) Solve the risk-averse stochastic optimal control 
                \State ~~~ problem \eqref{eq:stochOCPrisksensitive}. 
                \State 3.) Apply the MPC feedback $\mu_{\theta,N}(x_j) := U^*_{N,x_j}(0)$
                \State ~~~ to system \eqref{eq:stochSys}.
            \EndFor
        \end{algorithmic}
    \end{algorithm}

    Since the cost in \eqref{eq:stochOCPrisksensitive} does now satisfies the assumptions from \cite{Schiessl2024c} we can use the results from there to bound the performance of the MPC Algorithm~\ref{alg:stochMPCrisk}. For this, we have to make the following assumptions, where we define the distributional ball around a stationary process $\mathbf{X}^s$ with respect to a metric $d$ on the space of probability measures $\mathcal{P}(\X)$ as in Definition~\ref{defn:stochTurnpikeTypes}~(iii) as
    \begin{equation*}
        B^d_r(\mathbf{X}^s) := \{X \in \RR{\Omega,\X} \mid d(X,X^s(0)) < r \}.
    \end{equation*}
    Note that we use the notation $B^d_r(\mathbf{X}^s)$ since $d(X,X^s(0)) < r$ implies $d(X,X^s(k)) < r$ for all $k \in \N_0$ since $X^s(s) \sim X^s(k)$ holds for all $s,k \in \N_0$. 
    Similarly, we write $\ell(\mathbf{X}^s,\mathbf{U}^s)$ instead of $\ell(X^s(k),U^s(k))$ for law-invariant stage costs $\ell(X,U)$.
    Furthermore, to ensure that in addition to the stage costs $\ell(X,U)$ the costs $\ell_{\theta}(X,U)$ are finite we consider the constraint set
    \begin{equation*}
    \begin{split}
        \mathbb{X}_{\theta} := \{ X &\in \RR{\Omega, \X} \mid ~\exists~ U \in \RR{\Omega, \X} : \\
        &\vert \ell_{\theta}(X,U) \vert < \infty, ~f(X,U,W) \in \mathbb{X}_{\theta}\} \subseteq \XX.
    \end{split}
    \end{equation*}

    \begin{ass} \label{ass:performanceCDC2024}
        Consider the stochastic optimal control problem \eqref{eq:stochOCPrisksensitive} for some fixed $\theta \in \Theta$, a metric $d$ on the space of probability measures $\mathcal{P}(\X)$ and a stationary pair $(\mathbf{X}_{\theta}^s,\mathbf{U}_{\theta}^s)$ with $\vert \ell_{\theta}(\mathbf{X}_{\theta}^s,\mathbf{U}_{\theta}^s)) \vert < \infty$.
        Then, we assume that the following properties hold:
        \begin{enumerate}[(i)]
            \item There exists a set $\XX_0 \subseteq \XX_{\theta}$ and some $r > 0$ such that for the closed-loop trajectory $X_{\mu_{\theta,N}}$ generated by Algorithm~\ref{alg:stochMPCrisk} it holds that $X_{\mu_{\theta,N}}(k,X_0) \in B^d_r(\mathbf{X}_{\theta}^s) \subseteq \XX_{\theta}$ for all $k \in \N_0$, $X_0 \in \XX_0$, and $N \in \N$.
            \item The optimal control problem is optimally operated at $(\mathbf{X}_{\theta}^s,\mathbf{U}_{\theta}^s)$, 
            i.e.\ $$\liminf_{K \rightarrow \infty} \frac{1}{K} \sum_{k=0}^{K-1} \ell(X_{\mathbf{U}}(k,X_0), U(k)) - \ell(\mathbf{X}^s_{\theta},\mathbf{U}^s_{\theta}) \geq 0$$ holds for all $X_0 \in \XX$ and $\mathbf{U} \in \U^{\infty}(X_0)$.
            \item The optimal control problem has the finite and infinite distributional turnpike property from Definition~\ref{defn:stochTurnpikeTypes}~(iii) at $(\mathbf{X}_{\theta}^s,\mathbf{U}_{\theta}^s)$ on $B^d_r(\mathbf{X}_{\theta}^s)$ with respect to the metric $d$.
            \item The optimal control problem has a shifted value function, cf.\ \cite[Definition~9]{Schiessl2024c}, which is approximately distributionally continuous at $\mathbf{X}_{\theta}^s$ on finite and infinite horizons with respect to the metric $d$ according to \cite[Definition~10]{Schiessl2024c}.
        \end{enumerate}
    \end{ass}

    
    However, since these performance bounds are computed with respect to the new cost $\ell_{\theta}(X,U)$, it is not immediately clear how well the closed-loop solution of Algorithm~\ref{alg:stochMPCrisk} performs with respect to the original cost $\ell(X,U)$. So in the following we will analyze the performance of Algorithm~\ref{alg:stochMPCrisk} with respect to the original cost and show how we can achieve averaged performance optimality by making an appropriate choice of the parameter $\theta$ in problem \eqref{eq:stochOCPrisksensitive}.
    In order to this, we make the following additional assumption on the original stochastic optimal control problem \eqref{eq:stochOCP}.
    \begin{ass} \label{ass:optimallyOperated}
        The original stochastic optimal control problem \eqref{eq:stochOCP} is optimally operated at a stationary pair $(\mathbf{X}^s,\mathbf{U}^s)$.
    \end{ass}
    Now we can show the following result regarding the averaged closed-loop performance with respect to the original cost for the solution from Algorithm~\ref{alg:stochMPCrisk} 
    $\bar{J}^{cl}_{K}(X_0,\mu_{N,\theta}) :=  \frac{1}{K} J_K^{cl}(X_0,\mu_{N,\theta})$
    for arbitrary $\theta \in \Theta$ satisfying Assumption~\ref{ass:performanceCDC2024}.

    \begin{thm} \label{thm:averagedCost}
        Let the Assumptions~\ref{ass:mappingT},~\ref{ass:optimallyOperated} hold 
        and assume that
        \begin{equation} \label{eq:optimalParameter}
            \theta^s \in \argmin_{\theta \in \Theta} \ell_{\theta}(\mathbf{X}^s,\mathbf{U}^s),
        \end{equation}
        exists.
        Furthermore, consider $\theta \in \Theta$ and let Assumption~\ref{ass:performanceCDC2024} hold for this $\theta$.
        Then there exists $\delta \in \mathcal{L}$ and $\alpha \in \K_{\infty}$ such that the averaged closed-loop cost satisfies
        \begin{equation*}
        \begin{split}
            \ell(\mathbf{X}^s,\mathbf{U}^s) \leq \limsup_{K \to \infty}& \bar{J}^{cl}_{K}(X_0,\mu_{N,\theta}) \\
            \leq &\ell(\mathbf{X}^s,\mathbf{U}^s) + \delta(N) + \alpha(\Vert \theta - \theta^s \Vert)
        \end{split}
        \end{equation*}
        for all $X_0 \in \XX_0$.
    \end{thm}
    \begin{proof}
        By Assumption~\ref{ass:optimallyOperated} it follows immediately that 
        \begin{equation*}
            \liminf_{K \rightarrow \infty} \frac{1}{K} \sum_{k=0}^{K-1} \ell(X_{\mathbf{U}}(k,X_0), U(k)) \geq \ell(\mathbf{X}^s,\mathbf{U}^s)
        \end{equation*}
        holds for all $X_0 \in \XX_0$ and $\mathbf{U} \in \U^{\infty}(X_0)$, and thus, in particular $\ell(\mathbf{X}^s,\mathbf{U}^s) \leq \limsup_{K \to \infty} \bar{J}^{cl}_{K}(X_0,\mu_{N,\theta})$.
        Moreover, by \cite[Theorem~20]{Schiessl2024c} we know that 
        \begin{equation*}
            \begin{split}
                \limsup_{K \to \infty} \frac{1}{K} \sum_{k=0}^{K-1} \ell_{\theta}(X_{\mu}(k,X_0),\mu(X_{\mu_{N,\theta}}(k,X_0)))
                \\ \leq \ell_{\theta}(\mathbf{X}_{\theta}^s,\mathbf{U}_{\theta}^s) + \delta(N) =\ell(\mathbf{X}^s,\mathbf{U}^s) + \delta(N) + c
            \end{split}
        \end{equation*}
        holds for all $X_0 \in \XX_0$ with $c = \ell_{\theta}(\mathbf{X}_{\theta}^s,\mathbf{U}_{\theta}^s) - \ell(\mathbf{X}^s,\mathbf{U}^s)$. 
        Moreover, by the optimal operation from Assumption~\ref{ass:performanceCDC2024}~(ii) 
        we can conclude that $\ell_{\theta}(\mathbf{X}^s_{\theta},\mathbf{U}^s_{\theta}) \leq \ell_{\theta}(\mathbf{X}^s,\mathbf{U}^s)$ holds. Thus, by
        the uniform continuity of $\Psi$, cf.\ Assumption~\ref{ass:mappingT}~(ii), we get there is an $\alpha \in \K_{\infty}$ such that
        \begin{align*}
            &\ell_{\theta}(\mathbf{X}^s_{\theta},\mathbf{U}^s_{\theta}) - \ell_{\theta^s}(\mathbf{X}^s,\mathbf{U}^s) \\
            \leq& \ell_{\theta}(\mathbf{X}^s_{\theta},\mathbf{U}^s_{\theta}) - \ell_{\theta}(\mathbf{X}^s,\mathbf{U}^s) + \vert \ell_{\theta}(\mathbf{X}^s,\mathbf{U}^s) - \ell_{\theta^s}(\mathbf{X}^s,\mathbf{U}^s) \vert \\
            \leq& \Exp{\vert \Psi(g(\mathbf{X}^s,\mathbf{U}^s),\theta) - \Psi(g(\mathbf{X}^s,\mathbf{U}^s),\theta^s) \vert} 
            \leq \alpha(\Vert \theta - \theta^s\Vert)
        \end{align*}
        which shows the claim since $\ell(\mathbf{X}^s,\mathbf{U}^s) = \ell_{\theta^s}(\mathbf{X}^s,\mathbf{U}^s)$ holds due to the choice of $\theta^s$ from \eqref{eq:optimalParameter}.
    \end{proof}

    \begin{rem}
        \begin{enumerate}[(i)]
            \item In Assumption~\ref{ass:mappingT} we assume uniform continuity for all $z \in \R$. However, if we take a look at the proof of Theorem~\ref{thm:averagedCost}, we can see that it is sufficient that this condition holds for all $z \in \mbox{Im}(g(\mathbf{X}^s,\mathbf{U}^s)) \subseteq \R$. Hence, it is particularly fulfilled if $\Psi$ is continuous and the realizations of the stationary pair $(\mathbf{X}^s,\mathbf{U}^s)$ lie in a compact set.
            \item Note that Assumption~\ref{ass:performanceCDC2024}~(ii) is implied by Assumption~\ref{ass:optimallyOperated} for the optimal stationary parameter $\theta^s$ from \eqref{eq:optimalParameter}, but in general not for other values of $\theta$.
        \end{enumerate}
    \end{rem}

\section{Numerical Examples}

    Theorem~\ref{thm:averagedCost} shows that we can achieve near-optimality of averaged performance for the original cost criterion using Algorithm~\ref{alg:stochMPCrisk}, and that the distance to optimality is determined by the optimization horizon $N$ and by how well the fixed parameter $\theta$ fits the optimal stationary one $\theta^s$ from equation \eqref{eq:optimalParameter}. 
    Thus, our results show that switching from the expectation-based setting in \cite{Schiessl2024c} to the risk-averse setting considered in this paper extends the problem by a parameter-fitting component.
    In the following we will illustrate this result numerically.
    For this purpose we first consider the risk-averse setting from Example~\ref{ex:example2}. 
    
    To this end, we first have to find a good approximation of $\theta^s$ since equation~\eqref{eq:optimalParameter} is difficult to solve, both analytically and numerically.
    However, if we look at the parameters realizing the optimal costs in problem~\eqref{eq:stochOCP}, cf.\ Figure~\ref{fig:example2_parameters}, we can observe that these parameters exhibit a turnpike property analogously to the optimal states and controls.
    And since the turnpike property guarantees that we spend most of the time close to the stationary values, we can conclude that the values in the middle of the horizon should give us a good approximation of $\theta^s$ to fix the parameter $\theta$.

    \begin{figure}[t]
        \centering
        \includegraphics[width=0.3\textwidth]{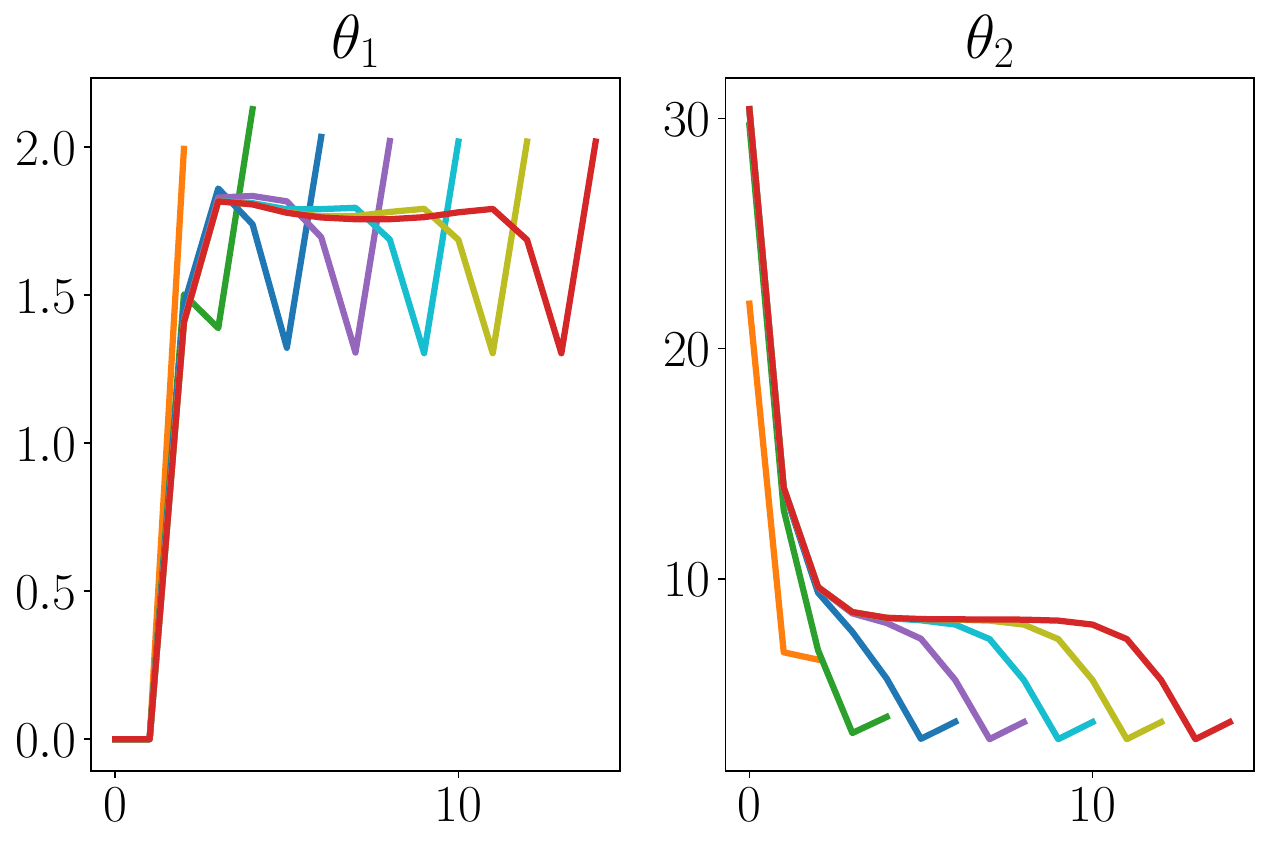}
        \caption{Parameter $\theta$ attaining the optimal cost for the optimal trajectories from Figure~\ref{fig:example2_turnpike}.}
        \label{fig:example2_parameters}
    \end{figure}
    
    The averaged performance results for different time horizons $N=6,7,8,9$ and parameter $\theta=\theta^s$ in Algorithm~\ref{alg:stochMPCrisk} are shown in Figure~\ref{fig:example2_costs}. In addition, Figure~\ref{fig:example2_costs} also shows the performance results for fixed horizon $N=9$ and different choices of the parameter $\theta = \theta^s, \theta_1, \theta_2$ with $\theta_1 = \theta^s + (1.5,1.5)$ and $\theta_2=\theta^s+(2.5,2.5)$.
    It is clearly observable that an increasing time horizon $N$ leads to a better average performance and that a deviation from the optimal stationary parameter $\theta^s$ leads to a poorer performance, which is in line with Theorem~\ref{thm:averagedCost}. 
    Moreover, we can see that the time horizon seems to have a much larger impact on the performance than the parameter $\theta$, which suggests that our approach is not too sensitive regarding an error in the parameter estimation, at least for this example.

    \begin{figure}[t]
        \centering
        \includegraphics[width=0.33\textwidth]{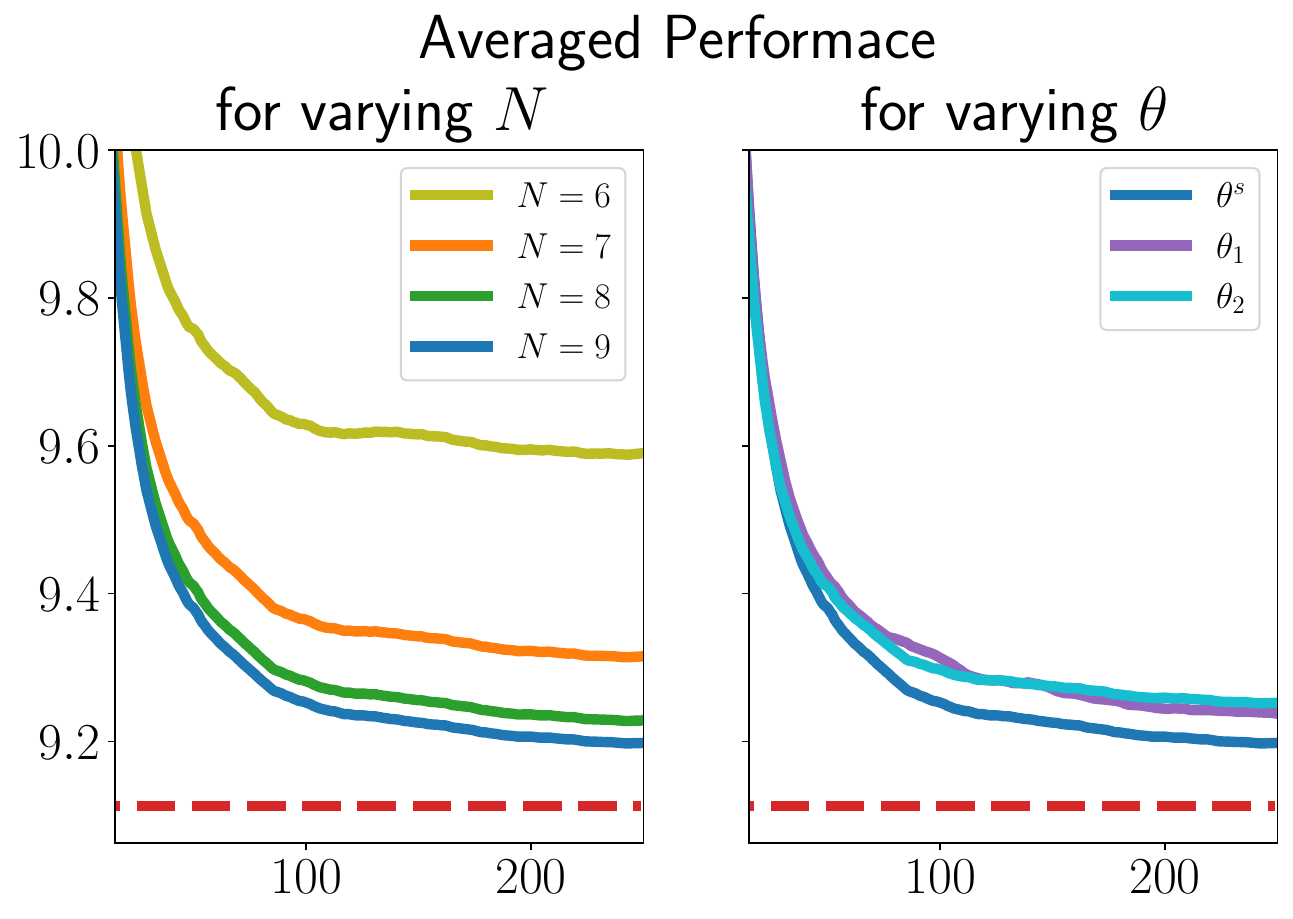}
        \caption{Optimal stationary cost $\ell(\mathbf{X}^s,\mathbf{U}^s)$ (red dashed) and averaged cost for $\theta^s$ on different horizons $N=6,7,8,9$ (left) and for horizon $N=9$ and different parameters $\theta^s,\theta_1,\theta_2$ (right) for the setting of Example~\ref{ex:example2}.}
        \label{fig:example2_costs}
    \end{figure}
    
    To further illustrate our theoretical findings, we implemented the MPC Algorithm~\ref{alg:stochMPCrisk} for the setting of Example~\ref{ex:example1}. 
    Again, we approximated the stationary optimal parameter $\theta^s$ and the corresponding optimal stationary cost by the turnpike. The obtained averaged performance results for varying horizons $N=6,7,8,9$ with fixed parameter $\theta = \theta^s$ and varying $\theta = \theta^s,\theta_1,\theta_2$ with $\theta_1 = 0.75 \theta^s$, $\theta_2=0.5 \theta^s$ and fixed $N=9$ are shown in Figure~\ref{fig:example1_costs}.
    We can observe that an increasing horizon $N$ leads to a better performance and a deviation from the parameter $\theta^s$ worsens the performance. 
    However, in contrast to Figure~\ref{fig:example2_costs}, we can see that for a too large deviation from the optimal parameter $\theta^s$ the performance significantly deteriorates. This shows that the parameter estimation is not negligible, although the sensitivity with respect to the parameter is again not too large.


\section{Conclusion} \label{sec:conclusions}
    We presented a risk-averse stochastic MPC algorithm with near-optimal averaged performance guarantees. 
    Our results are based on the observation that for parameterized risk measures, we can transfer results for stochastic MPC schemes with expected costs by fixing a certain parameter in the cost formulation. 
    The upper bound on the averaged performance then depends on how well this parameter matches an optimal one.
    Further research should focus on analyzing the non-averaged performance and stability properties of the proposed algorithm. 
    Furthermore, a useful extension for our algorithm would be to develop an adaptivity such that the fixed parameter $\theta$ is tuned automatically during the MPC loop and does not need to be calculated offline in advance.

    \begin{figure}[t]
        \centering
        \includegraphics[width=0.32\textwidth]{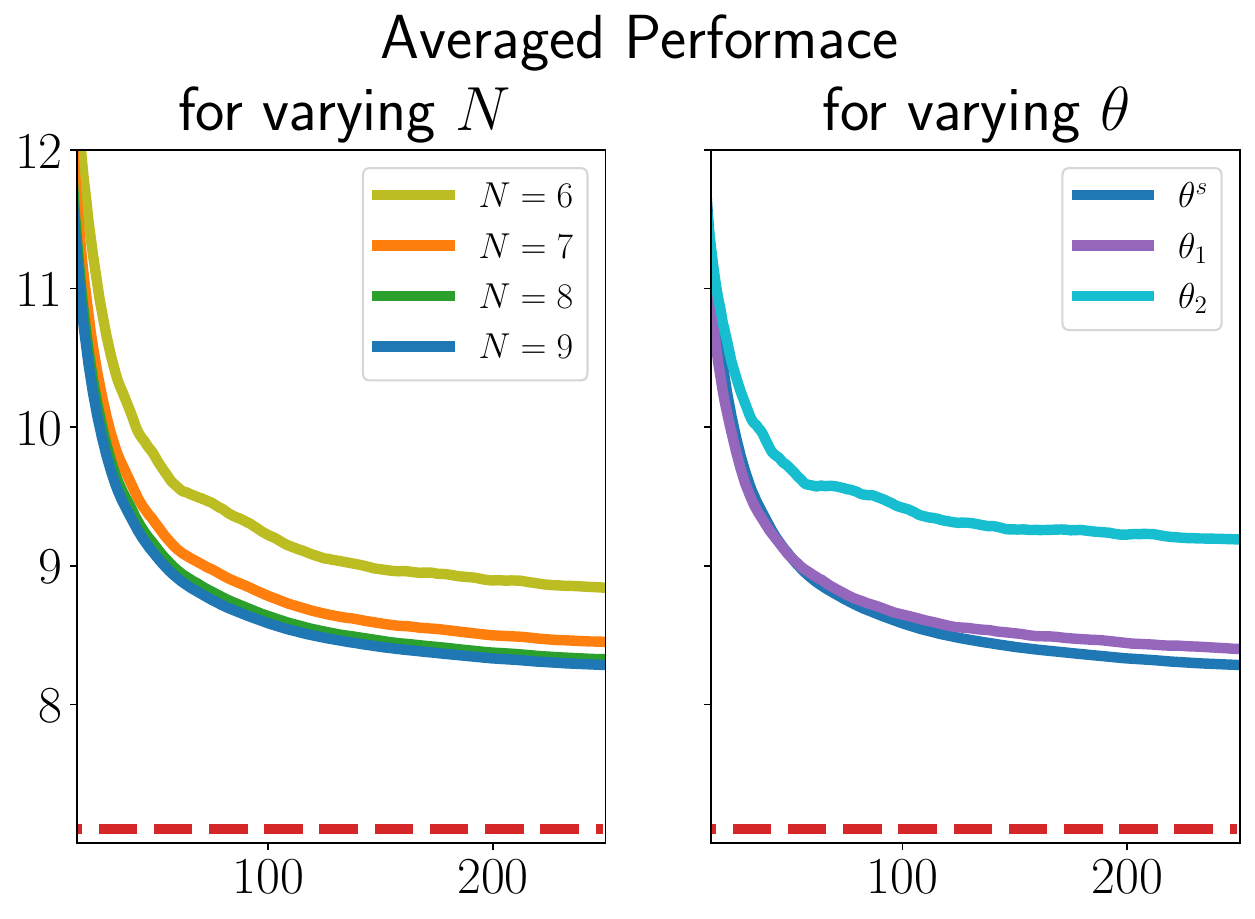}
        \caption{Optimal stationary cost $\ell(\mathbf{X}^s,\mathbf{U}^s)$ (red dashed) and averaged cost for $\theta^s$ on different horizons $N=6,7,8,9$ (left) and for horizon $N=9$ and different parameters $\theta^s,\theta_1,\theta_2$ (right) for the setting of Example~\ref{ex:example1}.}
        \label{fig:example1_costs}
    \end{figure}

{\bibliographystyle{abbrv} 
  \bibliography{references} 
}

\end{document}